\documentclass[oneside,10pt]{amsart}
\usepackage{amsfonts}
\usepackage{amsthm}
\usepackage{amsmath}
\usepackage{amscd}
\usepackage[latin2]{inputenc}
\usepackage{t1enc}
\usepackage[mathscr]{eucal}
\usepackage{indentfirst}
\usepackage{graphicx}
\usepackage{graphics}
\usepackage{pict2e}
\usepackage{epic}
\numberwithin{equation}{section}
\usepackage[margin=2.9cm]{geometry}
\usepackage{epstopdf}
\newtheorem{theorem}[subsection]{Theorem}
\newtheorem{lemma}[subsection]{Lemma}
\newtheorem{proposition}[subsection]{Proposition}

\newtheorem{corollary}[subsection]{Corollary}
\theoremstyle{definition}
\newtheorem{remark}[subsection]{Remark}
\newtheorem{definition}[subsection]{Definition}
\newtheorem{ex}[subsection]{Example}

\theoremstyle{property}

\makeatletter
\newcommand{\vast}{\bBigg@{2}}
\newcommand{\Vast}{\bBigg@{3}}
\makeatother

\begin{document}

\title{New Hermite-Hadamard and Simpson Type Inequalities For Harmonically $(s,m)$-convex functoins in Second Sense}
\author{Imran Abbas Baloch, $\dot{I}$mdat $\dot{I}$scan}
\address{Imran Abbas Baloch\\Abdus Salam School of Mathematical Sciences\\
GC University, Lahore, Pakistan}\email{iabbasbaloch@gmail.com\\ iabbasbaloch@sms.edu.pk }

\address{$\dot{I}$mdat Iscan\\ Department of Mathematics, Faculty of Arts and Sciences\\
 Giresun University, 28200, Giresun, TURKEY}
\email{imdat.iscan@giresun.edu.tr}



\begin{abstract} 
In \cite{II}, authors introduced the concept of harmonically $(s,m)$-convex functions in second sense which unifies different type of convexities and is more general notion of Harmonic convexity. In this paper, authors obtain new estimates on generalization of Hermite-Hadamard and Simpson type inequalities for this larger class of functions.
\end{abstract}

\maketitle
\section{\bf{Introduction}}
Let $f:I \subset \mathbb{R} \rightarrow \mathbb{R}$ be a convex function defined on the interval $I$ and $ a,b \in I$ with $a < b$, then following double inequalities hold
\begin{equation} \label{IQ1}
f\vast(  \frac{a + b}{2} \vast) \leq \frac{1}{b - a} \int_{a}^{b}f(x)dx \leq \frac{f(a) + f(b)}{2}.
\end{equation}
The inequality (\ref{IQ1}) is known in the literature as Hermite-Hadamard integral inequality.\\
 Let $f:[a,b] \rightarrow \mathbb{R}$ be a four times differentiable mapping on $(a,b)$ and $\|f^{(4)}\|_{\infty} = \sup_{x \in (a,b)} |f^{(4)}(x)| < \infty$, then the following inequality holds
\begin{equation}\label{IQ2}
\vast|\frac{1}{3} \vast[\frac{f(a) + f(b)}{2} + 2 f\vast(  \frac{a + b}{2} \vast) \vast] - \frac{1}{b - a} \int_{a}^{b}f(x)dx \vast| \leq \frac{1}{2880} \|f^{(4)}\|_{\infty} (b - a)^{4}.
\end{equation}
The inequality (\ref{IQ2}) is known in the literature as Simpson inequality. In recent years, many authors have studied errors estimates for Hermite-Hadamard and Simpson inequalities; for refinements, counterparts, generalization see [2,4,6,7,9,10].\\
In \cite{II}, authors introduced the concept of harmonically $(s,m)$-convex functions as follow
\begin{definition}
The function $f: I \subset (0, \infty) \rightarrow \mathbb{R}$ is said to be harmonically $(s,m)$-convex in second sense, where $s \in (0,1]$ and $m \in (0,1]$ if
$$f \big(\frac{mxy}{mty + (1 - t)x}\big) = f \big( (\frac{t}{x} + \frac{1 - t}{my})^{-1} \big) \leq t^{s} f(x) + m (1 - t)^{s} f(y)$$
$\forall x, y \in I$ and $t \in [0,1]$.
\end{definition}
\begin{remark}\label{IR1}
Note that for $s = 1$, harmonic $(s,m)$-convexity reduces to harmonic $m$-convexity and for $m = 1$, harmonic $(s,m)$-convexity reduces to harmonic $s$-convexity in second sense (see [5]) and for $s,m = 1$, harmonically $(s,m)$-convexity reduces to ordinary  harmonic convexity (see [4]).
\end{remark}
\section{\textbf{Some Basic Properties}}
In this section, we explore some basic results associated with harmonically $(s,m)$-convex functions in second sense.

\begin{proposition}\label{PP1}
Let $f:(0,\infty) \rightarrow \mathbb{R}$ be a function\\
a)  if $f$ is $(s,m)$-convex function in second sense and non-decreasing, then$f$ is harmonically $(s,m)$-convex function in second sense.\\
b)  if $f$ is harmonically $(s,m)$-convex function in second sense and non-increasing, then $f$ is $(s,m)$-convex function in second sense.\\
\end{proposition}
\begin{remark}
According to proposition \ref{PP1}, every non-decreasing $(s,m)$-convex function in second sense is also harmonically $(s,m)$-convex function in second sense.
\end{remark}
\begin{ex}(see\cite{BK})
Let $0 < s < 1$ and $a, b, c \in \mathbb{R}$, then function $f:(0,\infty) \rightarrow \mathbb{R}$ defined by
$$  f(x) = \left\{
                                                            \begin{array}{ll}
                                                              a, & \hbox{$ x = 0$} \\
                                                              b x^{s} + c, & \hbox{$x > 0$}
                                                            \end{array}
                                                          \right.    $$
is non-decreasing $s$-convex function in second sense for $ b \geq 0$ and $ 0 \leq c \leq a$. Hence, by proposition \ref{PP1}, $f$ is harmonically $(s,1)$-convex function.
\end{ex}

\begin{proposition}
Let $f:  (0, \infty) \rightarrow \mathbb{R}$ be a harmonically $(s,m)$-convex in second sense, where $s,m \in (0, 1]$ and let $a, b$ be nonnegative real numbers with $a < b$. Then for any $x \in [a,b]$, there is $t \in [0,1]$ such that
$$ f\vast( \frac{ab}{a + b - x} \vast) \leq t^{s} [ f(a) + f(b)] + m (1 - t)^{s} [f(\frac{a}{m}) + f(\frac{b}{m}) ]  - f\vast( \frac{ab}{x} \vast). $$
\end{proposition}
\begin{proof}
Since, any $x \in [a,b]$ can be represented as $ x  = ta + (1 - t)b$, $t \in [0,1]$, then
\begin{eqnarray*}
f \vast(\frac{ab}{a + b - x}\vast) &=& f \vast(\frac{ab}{a + b - ta - (1 - t)b}\vast)\\
&=& f \vast(\frac{ma(\frac{b}{m})}{ mt(\frac{b}{m}) + (1 - t)a }\vast)\\
&\leq& t^{s} f(a) + m(1 - t)^{s} f(\frac{b}{m})\\
&=& t^{s} f(a) + m(1 - t)^{s} f(\frac{b}{m}) + t^{s} f(b) - t^{s} f(b) + m(1 - t)^{s} f(\frac{a}{m}) - m(1 - t)^{s} f(\frac{a}{m})\\
&=& t^{s} [ f(a) + f(b)] + m (1 - t)^{s} [f(\frac{a}{m}) + f(\frac{b}{m}) ]  - f\vast( \frac{ab}{ta + (1 - t)b} \vast)\\
&=& t^{s} [ f(a) + f(b)] + m (1 - t)^{s} [f(\frac{a}{m}) + f(\frac{b}{m}) ]  - f\vast( \frac{ab}{x} \vast).
\end{eqnarray*}
\end{proof}
\begin{proposition}
Let $f_{i}:  (0, \infty) \rightarrow \mathbb{R}$, $i = 1,...,n$ are harmonically $(s,m)$-convex in second sense, where $s,m \in (0, 1]$, then function given by $f: = \max_{i = 1,...,n} \{f_{i}\}$ is also harmonically $(s,m)$-convex in second sense.
\end{proposition}
\begin{proposition}
Let $f_{n}:  (0, \infty) \rightarrow \mathbb{R}$ be a sequence of harmonically $(s,m)$-convex in second sense, where $s,m \in (0, 1]$ and $f_{n}(x) \rightarrow f(x)  $, then function  $f$ is also harmonically $(s,m)$-convex in second sense.
\end{proposition}
\begin{proposition}
Let $f:  (0, \infty) \rightarrow \mathbb{R}$ be a harmonically $(s_{1},m)$-convex in second sense and let $g:  (0, \infty) \rightarrow \mathbb{R}$ be a harmonically $(s_{2},m)$-convex in second sense, where $s_{1}, s_{2}, m \in (0, 1].$ Then $f + g $ is harmonically $(s,m)$-convex in second sense, where $ s = \min \{s_{1}, s_{2}\}$.
\end{proposition}

\begin{proposition}
Let $f:  (0, \infty) \rightarrow \mathbb{R}$ be a harmonically $(s,m)$-convex in second sense , where $s, m \in (0, 1].$ If $\lambda > 0$,  then $ \lambda f $ is harmonically $(s,m)$-convex in second sense.
\end{proposition}
\begin{proposition}
Let $f:  [0, b] \rightarrow \mathbb{R}$, $b > 0$ be a harmonically $m$-convex in second sense with $m \in (0, 1]$, and $g: I \subseteq f([0, b]) \rightarrow \mathbb{R}$ be nondecreasing and $(s,m)$-convex function in second sense on $I$ for some fixed $s$, then $gof$ is harmonically $(s,m)$-convex in second sense on $[0,b]$.
\end{proposition}
Let $f: I \subseteq (0,\infty) \rightarrow \mathbb{R}$ be a differentiable function on $I^{\circ}$, throughout this article we will assume that
$$ I_{f}(\lambda, \mu,a,b)= (\lambda - \mu) f\vast(  \frac{a + b}{2} \vast) + (1 - \lambda)f(a) + \mu f(b) - \frac{2ab}{b - a}\int_{a}^{b} \frac{f(u)}{u^{2}}du ,$$
where $a,b \in I$ with $ a < b$ and $\lambda, \mu \in \mathbb{R}.$\\
In \cite{Isk},  I.I\c{s}can et al established the following equality
\begin{lemma}\label{IL1}
Let $f: I \subseteq (0,\infty) \rightarrow \mathbb{R}$ be a differentiable function on $I^{\circ}$ such that $f' \in L[a,b]$, where $a,b \in I$ with $ a < b.$ Then for all $\lambda, \mu \in \mathbb{R}$, we have
$$ I_{f}(\lambda, \mu,a,b) = ab(b - a)\vast\{ \int_{0}^{\frac{1}{2}} \frac{\mu - t}{A^{2}_{t}} f'\vast(\frac{ab}{A_{t}}\vast)dt + \int^{1}_{\frac{1}{2}} \frac{\lambda - t}{A^{2}_{t}} f'\vast(\frac{ab}{A_{t}}\vast)dt \vast \},$$
where $A_{t} = tb + (1 - t)a$
\end{lemma}
In this paper, we establish more general form of Hermite-hadamard and Simpson type inequalities by using Lemma \ref{IL1} for harmonically $(s,m)$-convex functions.
\section{\textbf{Main Results}}
Now, we present our main results which are more general in the following section.
The Beta function, the Gamma function and the integral form of the hypergeometric function are defined as follows to be used in the sequel of paper
$$ B(\alpha,\beta) = \frac{\Gamma(\alpha) \Gamma(\beta)}{\Gamma(\alpha + \beta)} = \int_{0}^{1} t^{\alpha - 1} (1 - t)^{\beta - 1} dt,\; \alpha,\beta>0 $$

$$\Gamma(\alpha) = \int_{0}^{\infty} t^{\alpha - 1} e ^{-t} dt, \; \alpha > 0   $$
 and
 $$_{2}F_{1}(\alpha,\beta;\gamma,z) = \frac{1}{B(\beta,\gamma - \beta)}\int_{0}^{1} t^{\beta - 1} (1 - t)^{\gamma - \beta -1} (1 - zt)^{-\alpha} dt,\; \gamma>\beta>0,\;|z|<1  $$
\begin{theorem}\label{MT1}
Let $f:I \subset (0,\infty) \rightarrow \mathbb{R}$ be a differentiable function on $I^{\circ}$ such that $f' \in L[a,b]$, where $a,\frac{b}{m} \in I^{\circ}$ with $a < b$. If $|f'|^{q}$ is harmonically $(s,m)$-convex on $[a,\frac{b}{m}]$ for some fixed $q \geq 1$ and $ 0 \leq \mu \leq \frac{1}{2} \leq \lambda \leq 1$, then following inequality holds
\begin{eqnarray*}\label{MIE1}
\vast|I_{f}(\lambda,\mu,a,b)\vast| &\leq&  ab(b - a)\vast \{\mathcal{B}_{1}^{1 - \frac{1}{q}}(\mu)
\vast (|f'(a)|^{q} \mathcal{B}_{2}(\mu,q,a,b) + m|f'(\frac{b}{m})|^{q} \mathcal{B}_{3}(\mu,q,a,b) \vast)^{\frac{1}{q}}\\
&+& \mathcal{B}_{4}^{1 - \frac{1}{q}}(\lambda) \vast (|f'(a)|^{q} \mathcal{B}_{5}(\lambda,q,a,b) + m|f'(\frac{b}{m})|^{q} \mathcal{B}_{6}(\lambda,q,a,b) \vast) \vast\}
\end{eqnarray*}
where

$$
 \mu^{2} - \frac{\mu}{2} + \frac{1}{8} := \mathcal{B}_{1}(\mu)$$,
 $$ \lambda^{2} - \frac{3\lambda}{2} + \frac{5}{8}:= \mathcal{B}_{4}(\lambda),
$$

$$\mathcal{B}_{2}(\mu,q,a,b) = \left\{
                                                            \begin{array}{ll}
                                                              \frac{2^{2q - s- 2}\beta(1,s + 2)}{(a + b)^{2q}}.{}_{2}F_{1}(2q,1,s + 3,1 - \frac{2a}{b + a}), & \hbox{$    \mu = 0$} \\
                                                              {}\\
                                                               \frac{2 \mu^{s + 2}\beta(2,s + 1)}{[\mu b + (1 - \mu)a]^{2q}}.{}_{2}F_{1}(2q,2,s + 3,1 - \frac{a}{\mu b + (1 - \mu)a})
                                                                - \frac{\mu 2^{2q - s- 2}\beta(1,s + 1)}{(b + a)^{2q}}.{}_{2}F_{1}(2q,1,s + 2,1 - \frac{2a}{ b + a})\\
                                                                +\; \frac{ 2^{2q - s- 2}\beta(1,s + 2)}{(b + a)^{2q}}.{}_{2}F_{1}(2q,1,s + 3,1 - \frac{2a}{ b + a}),&\hbox{$0 < \mu < \frac{1}{2}$}\\
                                                               {}\\
                                                              \frac{2^{2q - s - 2}\beta(2,s + 1)}{(b + a)^{2q}}.{}_{2}F_{1}(2q,2,s + 3,1 - \frac{2a}{b + a}), &\hbox{$\mu = \frac{1}{2}$}
                                                            \end{array}
                                                          \right.  $$

$$\mathcal{B}_{3}(\mu,q,a,b) = \left\{
                                                            \begin{array}{ll}
                                                              \frac{\beta(s + 1, s + 3)}{b^{2q}}.{}_{2}F_{1}(2q,s + 1,s +3,1 - \frac{a}{b}) - \frac{\beta(s + 1,1)}{2^{s + 2}b^{2q}}.{}_{2}F_{1}(2q,s + 1,s + 2,1 - \frac{b + a}{2b})\\
                                                     - \frac{\beta(s + 1,2)}{2^{s + 2}b^{2q}}.{}_{2}F_{1}(2q,s + 1,s + 3,1 - \frac{b + a}{2b}), & \hbox{$    \mu = 0$} \\
                                                              {}\\
                                                               \frac{\mu \beta(s + 1,1)}{b^{2q}}.{}_{2}F_{1}(2q,s + 1,s + 2,1 - \frac{a}{b}) - \frac{\beta(s + 1,2)}{b^{2q}}.{}_{2}F_{1}(2q,s + 1,s + 3,1 - \frac{a}{b})\\
                            +\;2 \frac{(1 - \mu)^{s + 2} \beta(s + 1,2)}{b^{2q}}.{}_{2}F_{1}(2q,s + 1,s + 3,(1 - \mu)(1 - \frac{a}{b}))\\
                                                              +\; (\mu - 1)\frac{ \beta(s + 1,1)}{2^{s + 2}b^{2q}}.{}_{2}F_{1}(2q,s + 1,s + 2,1 - \frac{b + a}{2b})\\
                                                              + \;\frac{ \beta(s + 1,2)}{2^{s + 2}b^{2q}}.{}_{2}F_{1}(2q,s + 1,s + 3,1 - \frac{b + a}{2b}) ,&\hbox{$0 < \mu < \frac{1}{2}$}\\
                                                               {}\\
                                                              \frac{\beta(s + 1,1)}{2 b^{2q}}.{}_{2}F_{1}(2q,s + 2,s + 3,1 - \frac{a}{b })- \frac{\beta(s + 1,2)}{2 b^{2q}}.{}_{2}F_{1}(2q,s + 2,s + 3,1 - \frac{a}{b })\\
                                                              \frac{\beta(s + 1,2)}{2^{s + 2} b^{2q}}.{}_{2}F_{1}(2q,s + 2,s + 3,1 - \frac{b + a}{2b }), &\hbox{$\mu = \frac{1}{2}$}
                                                            \end{array}
                                                          \right.  $$
$$\mathcal{B}_{5}(\mu,q,a,b) = \left\{
                                                            \begin{array}{ll}
                                                              \frac{\beta(1,s + 2)}{b^{2q}}.{}_{2}F_{1}(2q,1,s + 3,1 - \frac{a}{b}) - \frac{2^{2q - s- 2}\beta(1,s + 2)}{(a + b)^{2q}}.{}_{2}F_{1}(2q,1,s + 3,1 - \frac{2a}{b + a}), & \hbox{$    \lambda = 0$} \\
                                                              {}\\
                                                               \frac{2 \lambda^{s + 2}\beta(2,s + 1)}{[\lambda b + (1 - \lambda)a]^{2q}}.{}_{2}F_{1}(2q,2,s + 3,1 - \frac{a}{\lambda b + (1 - \lambda)a})
                                                                - \frac{\lambda 2^{2q - s- 2}\beta(1,s + 1)}{(b + a)^{2q}}.{}_{2}F_{1}(2q,1,s + 2,1 - \frac{2a}{ b + a})\\
                                                                +\; \frac{ 2^{2q - s- 2}\beta(1,s + 2)}{(b + a)^{2q}}.{}_{2}F_{1}(2q,1,s + 3,1 - \frac{2a}{ b + a}) + \frac{\beta(1,s + 2)}{b^{2q}}.{}_{2}F_{1}(2q,1,s + 3,1 - \frac{a}{b})\\
                                                                - \frac{\lambda \beta(1,s + 1)}{b^{2q}}.{}_{2}F_{1}(2q,1,s + 2,1 - \frac{a}{b}),&\hbox{$0 < \lambda < \frac{1}{2}$}\\
                                                               {}\\
                                                              \frac{\beta(1,s + 2)}{2b^{2q}}.{}_{2}F_{1}(2q,1,s + 3,1 - \frac{a}{b})  + \frac{2^{2q - s - 2}\beta(2,s + 1)}{(b + a)^{2q}}.{}_{2}F_{1}(2q,2,s + 3,1 - \frac{2a}{b + a})\\
                                                             - \frac{\beta(2,s + 1)}{2b^{2q}}.{}_{2}F_{1}(2q,2,s + 3,1 - \frac{a}{b}), &\hbox{$\lambda = \frac{1}{2}$}
                                                            \end{array}
                                                          \right.  $$

$$\mathcal{B}_{6}(\mu,q,a,b) = \left\{
                                                            \begin{array}{ll}
                                                            \frac{\beta(s + 1,1)}{2^{s + 2}b^{2q}}.{}_{2}F_{1}(2q,s + 1,s + 2,1 - \frac{b + a}{2b})
                                                 - \frac{\beta(s + 1,2)}{2^{s + 2}b^{2q}}.{}_{2}F_{1}(2q,s + 1,s + 3,1 - \frac{b + a}{2b}), & \hbox{$    \lambda = 0$} \\
                                                              {}\\
                                                               \frac{2\lambda \beta(s + 1,1)}{b^{2q}}.{}_{2}F_{1}(2q,s + 1,s + 2,1 - \frac{a}{b})
                                                              +\; (\lambda - 1)\frac{ \beta(1,s + 1)}{2^{s + 1}b^{2q}}.{}_{2}F_{1}(2q, 1,s + 2,1 - \frac{b + a}{2b})\\
                                                             +\;2 \frac{(1 - \lambda)^{s + 2} \beta(s + 1,2)}{b^{2q}}.{}_{2}F_{1}(2q,s + 1,s + 3,(1 - \lambda)(1 - \frac{a}{b}))\\
                                                              + \;\frac{ \beta(2,s + 1)}{2^{s + 1}b^{2q}}.{}_{2}F_{1}(2q,2,s + 3,1 - \frac{b + a}{2b}) ,&\hbox{$0 <  \lambda < \frac{1}{2}$}\\
                                                               {}\\
                                                              \frac{\beta(s + 1,2)}{2^{s + 2} b^{2q}}.{}_{2}F_{1}(2q,s + 2,s + 3,1 - \frac{b + a}{2b }), &\hbox{$\lambda = \frac{1}{2}$}
                                                            \end{array}
                                                          \right.  $$
\end{theorem}

\begin{proof}
using Lemma \ref{IL1}, H$\ddot{o}$lder's inequality and harmonically $(s,m)$-convexity in second sense of $|f'|^{q}$, we get

\begin{eqnarray*}
\vast|I_{f}(\lambda,\mu,a,b)\vast| &\leq & ab(b - a)\vast\{ \vast( \int^{\frac{1}{2}}_{0} |\mu - t|dt \vast)^{1 - \frac{1}{q}} \vast (\int^{\frac{1}{2}}_{0} \frac{|\mu - t|}{A_{t}^{2q}} \big|f'\vast( \frac{ab}{A_{t}} \vast)\big|^{q}   \vast)^{\frac{1}{q}}\\
&+& \vast( \int^{1}_{\frac{1}{2}} |\lambda - t|dt \vast)^{1 - \frac{1}{q}} \vast (\int^{1}_{\frac{1}{2}} \frac{|\lambda - t|}{A_{t}^{2q}} \big|f'\vast( \frac{ab}{A_{t}} \vast)\big|^{q}   \vast)^{\frac{1}{q}} \vast\}\\
&\leq& ab(b - a)\vast\{ \vast( \int^{\frac{1}{2}}_{0} |\mu - t|dt \vast)^{1 - \frac{1}{q}}\\
&\times& \vast (|f'(a)|^{q} \int^{\frac{1}{2}}_{0} \frac{|\mu - t|t^{s}}{A_{t}^{2q}}dt + m|f'(\frac{b}{m})|^{q} \int^{\frac{1}{2}}_{0} \frac{|\mu - t| (1 - t)^{s}}{A_{t}^{2q}} \vast)^{\frac{1}{q}}\\
&+& \vast( \int^{1}_{\frac{1}{2}} |\lambda - t|dt \vast)^{1 - \frac{1}{q}} \\
&\times& \vast (|f'(a)|^{q} \int^{1}_{\frac{1}{2}} \frac{|\lambda - t|t^{s}}{A_{t}^{2q}}dt + m|f'(\frac{b}{m})|^{q} \int^{1}_{\frac{1}{2}} \frac{|\lambda - t| (1 - t)^{s}}{A_{t}^{2q}} \vast)^{\frac{1}{q}},\\
\end{eqnarray*}
where, by calculations we find that

$$
\int^{\frac{1}{2}}_{0} |\mu - t|dt = \mu^{2} - \frac{\mu}{2} + \frac{1}{8},
 \;\;\;\int^{1}_{\frac{1}{2}} |\lambda - t|dt = \lambda^{2} - \frac{3\lambda}{2} + \frac{5}{8},
$$

$$\int^{\frac{1}{2}}_{0} \frac{|\mu - t|t^{s}}{A_{t}^{2q}}dt = \left\{
                                                            \begin{array}{ll}
                                                              \frac{2^{2q - s- 2}\beta(1,s + 2)}{(a + b)^{2q}}.{}_{2}F_{1}(2q,1,s + 3,1 - \frac{2a}{b + a}), & \hbox{$    \mu = 0$} \\
                                                              {}\\
                                                               \frac{2 \mu^{s + 2}\beta(2,s + 1)}{[\mu b + (1 - \mu)a]^{2q}}.{}_{2}F_{1}(2q,2,s + 3,1 - \frac{a}{\mu b + (1 - \mu)a})
                                                                - \frac{\mu 2^{2q - s- 2}\beta(1,s + 1)}{(b + a)^{2q}}.{}_{2}F_{1}(2q,1,s + 2,1 - \frac{2a}{ b + a})\\
                                                                +\; \frac{\mu 2^{2q - s- 2}\beta(1,s + 2)}{(b + a)^{2q}}.{}_{2}F_{1}(2q,1,s + 3,1 - \frac{2a}{ b + a}),&\hbox{$0 < \mu < \frac{1}{2}$}\\
                                                               {}\\
                                                              \frac{2^{2q - s - 2}\beta(2,s + 1)}{(b + a)^{2q}}.{}_{2}F_{1}(2q,2,s + 3,1 - \frac{2a}{b + a}), &\hbox{$\mu = \frac{1}{2}$}
                                                            \end{array}
                                                          \right.  $$

$$\int^{\frac{1}{2}}_{0} \frac{|\mu - t|(1 - t)^{s}}{A_{t}^{2q}}dt = \left\{
                                                            \begin{array}{ll}
                                                              \frac{\beta(s + 1, s + 3)}{b^{2q}}.{}_{2}F_{1}(2q,s + 1,s +3,1 - \frac{a}{b}) - \frac{\beta(s + 1,1)}{2^{s + 2}b^{2q}}.{}_{2}F_{1}(2q,s + 1,s + 2,1 - \frac{b + a}{2b})\\
                                                     - \frac{\beta(s + 1,2)}{2^{s + 2}b^{2q}}.{}_{2}F_{1}(2q,s + 1,s + 3,1 - \frac{b + a}{2b}), & \hbox{$    \mu = 0$} \\
                                                              {}\\
                                                               \frac{\mu \beta(s + 1,1)}{b^{2q}}.{}_{2}F_{1}(2q,s + 1,s + 2,1 - \frac{a}{b}) - \frac{\beta(s + 1,2)}{b^{2q}}.{}_{2}F_{1}(2q,s + 1,s + 3,1 - \frac{a}{b})\\
                            +\;2 \frac{(1 - \mu)^{s + 2} \beta(s + 1,2)}{b^{2q}}.{}_{2}F_{1}(2q,s + 1,s + 3,(1 - \mu)(1 - \frac{a}{b}))\\
                                                              +\; (\mu - 1)\frac{ \beta(s + 1,1)}{2^{s + 2}b^{2q}}.{}_{2}F_{1}(2q,s + 1,s + 2,1 - \frac{b + a}{2b})\\
                                                              + \;\frac{ \beta(s + 1,2)}{2^{s + 2}b^{2q}}.{}_{2}F_{1}(2q,s + 1,s + 3,1 - \frac{b + a}{2b}) ,&\hbox{$0 < \mu < \frac{1}{2}$}\\
                                                               {}\\
                                                              \frac{\beta(s + 1,1)}{2 b^{2q}}.{}_{2}F_{1}(2q,s + 2,s + 3,1 - \frac{a}{b })- \frac{\beta(s + 1,2)}{2 b^{2q}}.{}_{2}F_{1}(2q,s + 2,s + 3,1 - \frac{a}{b })\\
                                                              \frac{\beta(s + 1,2)}{2^{s + 2} b^{2q}}.{}_{2}F_{1}(2q,s + 2,s + 3,1 - \frac{b + a}{2b }), &\hbox{$\mu = \frac{1}{2}$}
                                                            \end{array}
                                                          \right.  $$
$$\int_{\frac{1}{2}}^{1} \frac{|\lambda - t|t^{s}}{A_{t}^{2q}}dt = \left\{
                                                            \begin{array}{ll}
                                                              \frac{\beta(1,s + 2)}{b^{2q}}.{}_{2}F_{1}(2q,1,s + 3,1 - \frac{a}{b}) - \frac{2^{2q - s- 2}\beta(1,s + 2)}{(a + b)^{2q}}.{}_{2}F_{1}(2q,1,s + 3,1 - \frac{2a}{b + a}), & \hbox{$    \lambda = 0$} \\
                                                              {}\\
                                                               \frac{2 \lambda^{s + 2}\beta(2,s + 1)}{[\lambda b + (1 - \lambda)a]^{2q}}.{}_{2}F_{1}(2q,2,s + 3,1 - \frac{a}{\lambda b + (1 - \lambda)a})
                                                                - \frac{\lambda 2^{2q - s- 2}\beta(1,s + 1)}{(b + a)^{2q}}.{}_{2}F_{1}(2q,1,s + 2,1 - \frac{2a}{ b + a})\\
                                                                +\; \frac{ 2^{2q - s- 2}\beta(1,s + 2)}{(b + a)^{2q}}.{}_{2}F_{1}(2q,1,s + 3,1 - \frac{2a}{ b + a}) + \frac{\beta(1,s + 2)}{b^{2q}}.{}_{2}F_{1}(2q,1,s + 3,1 - \frac{a}{b})\\
                                                                - \frac{\lambda \beta(1,s + 1)}{b^{2q}}.{}_{2}F_{1}(2q,1,s + 2,1 - \frac{a}{b}),&\hbox{$0 < \lambda < \frac{1}{2}$}\\
                                                               {}\\
                                                              \frac{\beta(1,s + 2)}{2b^{2q}}.{}_{2}F_{1}(2q,1,s + 3,1 - \frac{a}{b})  + \frac{2^{2q - s - 2}\beta(2,s + 1)}{(b + a)^{2q}}.{}_{2}F_{1}(2q,2,s + 3,1 - \frac{2a}{b + a})\\
                                                             - \frac{\beta(2,s + 1)}{2b^{2q}}.{}_{2}F_{1}(2q,2,s + 3,1 - \frac{a}{b}), &\hbox{$\lambda = \frac{1}{2}$}
                                                            \end{array}
                                                          \right.  $$

$$\int_{\frac{1}{2}}^{1} \frac{|\lambda - t|(1 - t)^{s}}{A_{t}^{2q}}dt = \left\{
                                                            \begin{array}{ll}
                                                            \frac{\beta(s + 1,1)}{2^{s + 2}b^{2q}}.{}_{2}F_{1}(2q,s + 1,s + 2,1 - \frac{b + a}{2b})
                                                 - \frac{\beta(s + 1,2)}{2^{s + 2}b^{2q}}.{}_{2}F_{1}(2q,s + 1,s + 3,1 - \frac{b + a}{2b}), & \hbox{$    \lambda = 0$} \\
                                                              {}\\
                                                               \frac{2\lambda \beta(s + 1,1)}{b^{2q}}.{}_{2}F_{1}(2q,s + 1,s + 2,1 - \frac{a}{b})
                                                              +\; (\lambda - 1)\frac{ \beta(1,s + 1)}{2^{s + 1}b^{2q}}.{}_{2}F_{1}(2q, 1,s + 2,1 - \frac{b + a}{2b})\\
                                                             +\;2 \frac{(1 - \lambda)^{s + 2} \beta(s + 1,2)}{b^{2q}}.{}_{2}F_{1}(2q,s + 1,s + 3,(1 - \lambda)(1 - \frac{a}{b}))\\
                                                              + \;\frac{ \beta(2,s + 1)}{2^{s + 1}b^{2q}}.{}_{2}F_{1}(2q,2,s + 3,1 - \frac{b + a}{2b}) ,&\hbox{$0 <  \lambda < \frac{1}{2}$}\\
                                                               {}\\
                                                              \frac{\beta(s + 1,2)}{2^{s + 2} b^{2q}}.{}_{2}F_{1}(2q,s + 2,s + 3,1 - \frac{b + a}{2b }), &\hbox{$\lambda = \frac{1}{2}$}
                                                            \end{array}
                                                          \right.  $$

which completes the proof.
\end{proof}
\begin{corollary}
Under the assumption of Theorem \ref{MT1} with $ \lambda = \mu = \frac{1}{2}$, the inequality (\ref{MIE1}) reduced to to the following inequality
\begin{eqnarray*}
\vast| \frac{f(a) + f(b)}{2} - \frac{ab}{b - a}\int_{a}^{b} \frac{f(u)}{u^{2}}du \vast| &\leq&  ab(b - a)\vast( \frac{1}{8} \vast)^{1 - \frac{1}{q}}\vast \{
\vast (|f'(a)|^{q} \mathcal{B}_{2}(\frac{1}{2},q,a,b) + m|f'(\frac{b}{m})|^{q} \mathcal{B}_{3}(\frac{1}{2},q,a,b) \vast)^{\frac{1}{q}}\\
&+& \vast (|f'(a)|^{q} \mathcal{B}_{5}(\frac{1}{2},q,a,b) + m|f'(\frac{b}{m})|^{q} \mathcal{B}_{6}(\frac{1}{2},q,a,b) \vast) \vast\}
\end{eqnarray*}
\end{corollary}
\begin{corollary}
Under the assumption of Theorem \ref{MT1} with $ \lambda = 1 $ and $ \mu = 0$, the inequality (\ref{MIE1}) reduced to to the following inequality
\begin{eqnarray*}
\vast| f ( \frac{2ab}{a + b} ) - \frac{ab}{b - a}\int_{a}^{b} \frac{f(u)}{u^{2}}du \vast|  &\leq&  ab(b - a)\vast( \frac{1}{8} \vast)^{1 - \frac{1}{q}}\vast \{
\vast (|f'(a)|^{q} \mathcal{B}_{2}(0,q,a,b) + m|f'(\frac{b}{m})|^{q} \mathcal{B}_{3}(0,q,a,b) \vast)^{\frac{1}{q}}\\
&+& \vast (|f'(a)|^{q} \mathcal{B}_{5}(1,q,a,b) + m|f'(\frac{b}{m})|^{q} \mathcal{B}_{6}(1,q,a,b) \vast) \vast\}
\end{eqnarray*}
\end{corollary}
\begin{corollary}
Under the assumption of Theorem \ref{MT1} with $ \lambda = \frac{5}{6} $ and $\mu = \frac{1}{6}$, the inequality (\ref{MIE1}) reduced to to the following inequality
$$
\vast|\frac{1}{3} \vast[\frac{f(a) + f(b)}{2} + 2 f\vast(  \frac{2ab}{a + b} \vast) \vast] - \frac{ab}{b - a} \int_{a}^{b}\frac{f(u)}{u^{2}}du \vast|$$
$$\leq  ab(b - a)\vast( \frac{5}{72} \vast)^{1 - \frac{1}{q}}\vast \{
\vast (|f'(a)|^{q} \mathcal{B}_{2}(\frac{1}{6},q,a,b) + m|f'(\frac{b}{m})|^{q} \mathcal{B}_{3}(\frac{1}{6},q,a,b) \vast)^{\frac{1}{q}}$$
$$+ \vast (|f'(a)|^{q} \mathcal{B}_{5}(\frac{5}{6},q,a,b) + m|f'(\frac{b}{m})|^{q} \mathcal{B}_{6}(\frac{5}{6},q,a,b) \vast) \vast\}$$
\end{corollary}

\begin{theorem}\label{MT2}
Let $f:I \subset (0,\infty) \rightarrow \mathbb{R}$ be a differentiable function on $I^{\circ}$ such that $f' \in L[a,b]$, where $a,\frac{b}{m} \in I^{\circ}$ with $a < b$. If $|f'|^{q}$ is harmonically $(s,m)$-convex on $[a,\frac{b}{m}]$ for some fixed $q > 1$ and $ 0 \leq \mu \leq \frac{1}{2} \leq \lambda \leq 1$, then following inequality holds
\begin{eqnarray*}\label{MIE2}
\vast|I_{f}(\lambda,\mu,a,b)\vast| &\leq&  ab(b - a)\vast \{\mathcal{B}_{7}^{\frac{1}{p}}(\mu)\vast (|f'(a)|^{q} \mathcal{B}_{8}(\mu,q,a,b) + m|f'(\frac{b}{m})|^{q} \mathcal{B}_{9}(\mu,q,a,b) \vast)^{\frac{1}{q}}\\ &+& \mathcal{B}_{10}^{\frac{1}{p}}(\lambda) \vast (|f'(a)|^{q} \mathcal{B}_{11}(\lambda,q,a,b) + m|f'(\frac{b}{m})|^{q} \mathcal{B}_{12}(\lambda,q,a,b) \vast)^{\frac{1}{q}} \vast\}
\end{eqnarray*}
where
$$
\mathcal{B}_{7}(\mu) = \frac{1}{p + 1} \vast[ \mu^{p + 1} + \vast( \frac{1}{2} - \mu \vast)^{p + 1} \vast] ,$$
$$\mathcal{B}_{10}(\lambda) + \frac{1}{p + 1} \vast[ (\lambda - \frac{1}{2})^{p + 1} + \vast( 1 - \lambda \vast)^{p + 1} \vast], $$
$$\mathcal{B}_{8}(\mu,q,a,b) = \frac{2^{2q - s -2}\beta(1,s + 1)}{(b + a)^{2q}}.{}_{2}F_{1}(2q,1,s + 2,1 - \frac{2a}{b + a})$$
$$\mathcal{B}_{9}(\mu,q,a,b) = \frac{\beta(s + 1,1)}{b^{2q}}.{}_{2}F_{1}(2q,s + 1,s + 2,1 - \frac{a}{b}) - \frac{\beta(s + 1, 2)}{2^{s + 1}b^{2q}}.{}_{2}F_{1}(2q,s + 1,s + 2,1 - \frac{a + b}{2b })$$
$$\mathcal{B}_{11}(\lambda,q,a,b) = \frac{\beta(1,s + 1)}{b^{2q}}.{}_{2}F_{1}(2q,1,s + 2,1 - \frac{a}{b}) - \frac{2^{2q - s - 1}\beta(1,s + 1)}{(b + a)^{2q}}.{}_{2}F_{1}(2q,1,s + 2,1 - \frac{2a}{b + a })$$

$$\mathcal{B}_{12}(\mu,q,a,b) =  \frac{\beta(s + 1, 2)}{2^{s + 1}b^{2q}}.{}_{2}F_{1}(2q,s + 1,s + 3,1 - \frac{a + b}{2b })$$
\end{theorem}
\begin{proof}
using Lemma \ref{IL1}, H$\ddot{o}$lder's inequality and harmonically $(s,m)$-convexity in second sense of $|f'|^{q}$, we get

\begin{eqnarray*}
\vast|I_{f}(\lambda,\mu,a,b)\vast| &\leq & ab(b - a)\vast\{ \vast( \int^{\frac{1}{2}}_{0} |\mu - t|^{p}dt \vast)^{\frac{1}{p}} \vast (\int^{\frac{1}{2}}_{0} \frac{1}{A_{t}^{2q}} \big|f'\vast( \frac{ab}{A_{t}} \vast)\big|^{q}   \vast)^{\frac{1}{q}}\\
&+& \vast( \int^{1}_{\frac{1}{2}} |\lambda - t|^{p}dt \vast)^{\frac{1}{p}} \vast (\int^{1}_{\frac{1}{2}} \frac{1}{A_{t}^{2q}} \big|f'\vast( \frac{ab}{A_{t}} \vast)\big|^{q}   \vast)^{\frac{1}{q}} \vast\}\\
&\leq& ab(b - a)\vast\{ \vast( \int^{\frac{1}{2}}_{0} |\mu - t|^{p}dt \vast)^{\frac{1}{p}}\\
&\times& \vast (|f'(a)|^{q} \int^{\frac{1}{2}}_{0} \frac{t^{s}}{A_{t}^{2q}}dt + m|f'(\frac{b}{m})|^{q} \int^{\frac{1}{2}}_{0} \frac{ (1 - t)^{s}}{A_{t}^{2q}} \vast)^{\frac{1}{q}}dt\\
&+& \vast( \int^{1}_{\frac{1}{2}} |\lambda - t|^{p}dt \vast)^{\frac{1}{p}} \\
&\times& \vast (|f'(a)|^{q} \int^{1}_{\frac{1}{2}} \frac{t^{s}}{A_{t}^{2q}}dt + m|f'(\frac{b}{m})|^{q} \int^{1}_{\frac{1}{2}} \frac{ (1 - t)^{s}}{A_{t}^{2q}} \vast)^{\frac{1}{q}},\\
\end{eqnarray*}
where, by calculations we find that
$$\vast( \int^{\frac{1}{2}}_{0} |\mu - t|^{p}dt \vast)^{\frac{1}{p}} = \frac{1}{p + 1} \vast[ \mu^{p + 1} + \vast( \frac{1}{2} - \mu \vast)^{p + 1} \vast]$$
$$ \vast( \int^{1}_{\frac{1}{2}} |\lambda - t|^{p}dt \vast)^{\frac{1}{p}} = \frac{1}{p + 1} \vast[ (\lambda - \frac{1}{2})^{p + 1} + \vast( 1 - \lambda \vast)^{p + 1} \vast] $$
$$ \int^{\frac{1}{2}}_{0} \frac{t^{s}}{A_{t}^{2q}}dt = \frac{2^{2q - s -2}\beta(1,s + 1)}{(b + a)^{2q}}.{}_{2}F_{1}(2q,1,s + 2,1 - \frac{2a}{b + a})$$
$$\int^{\frac{1}{2}}_{0} \frac{ (1 - t)^{s}}{A_{t}^{2q}} dt = \frac{\beta(s + 1,1)}{b^{2q}}.{}_{2}F_{1}(2q,s + 1,s + 2,1 - \frac{a}{b}) - \frac{\beta(s + 1, 2)}{2^{s + 1}b^{2q}}.{}_{2}F_{1}(2q,s + 1,s + 2,1 - \frac{a + b}{2b })  $$
$$ \int^{1}_{\frac{1}{2}} \frac{t^{s}}{A_{t}^{2q}}dt = \frac{\beta(1,s + 1)}{b^{2q}}.{}_{2}F_{1}(2q,1,s + 2,1 - \frac{a}{b}) - \frac{2^{2q - s - 1}\beta(1,s + 1)}{(b + a)^{2q}}.{}_{2}F_{1}(2q,1,s + 2,1 - \frac{2a}{b + a })$$
$$ \int^{1}_{\frac{1}{2}} \frac{ (1 - t)^{s}}{A_{t}^{2q}} =  \frac{\beta(s + 1, 2)}{2^{s + 1}b^{2q}}.{}_{2}F_{1}(2q,s + 1,s + 3,1 - \frac{a + b}{2b }) $$
which completes the proof.
\end{proof}
\begin{corollary}
Under the assumption of Theorem \ref{MT2} with $ \lambda = \mu = \frac{1}{2}$, the inequality (\ref{MIE2}) reduced to to the following inequality
\begin{eqnarray*}
\vast| \frac{f(a) + f(b)}{2} - \frac{ab}{b - a}\int_{a}^{b} \frac{f(u)}{u^{2}}du \vast| &\leq& ab(b - a) \vast( \frac{1}{(p + 1)2^{p + 1}} \vast)^{\frac{1}{p}} \vast \{\vast (|f'(a)|^{q} \mathcal{B}_{8}(q,a,b) + m|f'(\frac{b}{m})|^{q} \mathcal{B}_{9}(q,a,b) \vast)^{\frac{1}{q}}\\ &+&  \vast (|f'(a)|^{q} \mathcal{B}_{11}(q,a,b) + m|f'(\frac{b}{m})|^{q} \mathcal{B}_{12}(q,a,b) \vast)^{\frac{1}{q}} \vast\}
\end{eqnarray*}
\end{corollary}
\begin{corollary}
Under the assumption of Theorem \ref{MT2} with $ \lambda = 1 $ and $ \mu = 0$, the inequality (\ref{MIE2}) reduced to to the following inequality
\begin{eqnarray*}
\vast| f ( \frac{2ab}{a + b} ) - \frac{ab}{b - a}\int_{a}^{b} \frac{f(u)}{u^{2}}du \vast|  &\leq& ab(b - a) \vast( \frac{1}{(p + 1)2^{p + 1}} \vast)^{\frac{1}{p}} \vast \{\vast (|f'(a)|^{q} \mathcal{B}_{8}(q,a,b) + m|f'(\frac{b}{m})|^{q} \mathcal{B}_{9}(q,a,b) \vast)^{\frac{1}{q}}\\ &+&  \vast (|f'(a)|^{q} \mathcal{B}_{11}(q,a,b) + m|f'(\frac{b}{m})|^{q} \mathcal{B}_{12}(q,a,b) \vast)^{\frac{1}{q}} \vast\}
\end{eqnarray*}
\end{corollary}

\begin{corollary}
Under the assumption of Theorem \ref{MT2} with $ \lambda = \frac{5}{6} $ and $\mu = \frac{1}{6}$, the inequality (\ref{MIE2}) reduced to to the following inequality
$$
\vast|\frac{1}{3} \vast[\frac{f(a) + f(b)}{2} + 2 f\vast(  \frac{2ab}{a + b} \vast) \vast] - \frac{ab}{b - a} \int_{a}^{b}\frac{f(u)}{u^{2}}du \vast|$$
$$\leq  ab(b - a)\vast( \frac{(2^{p + 1})}{(p + 1)6^{p + 1}} \vast)^{\frac{1}{p}}\vast \{
\vast (|f'(a)|^{q} \mathcal{B}_{2}(,q,a,b) + m|f'(\frac{b}{m})|^{q} \mathcal{B}_{3}(q,a,b) \vast)^{\frac{1}{q}}$$
$$+ \vast (|f'(a)|^{q} \mathcal{B}_{5}(q,a,b) + m|f'(\frac{b}{m})|^{q} \mathcal{B}_{6}(q,a,b) \vast) \vast\}$$
\end{corollary}

\end{document}